\documentclass[11pt,english]{amsart}
\PassOptionsToPackage{obeyFinal}{todonotes}
\usepackage[T1]{fontenc}
\usepackage[latin9]{inputenc}
\usepackage{color}
\usepackage{babel}
\usepackage{todonotes}
\usepackage{amsbsy}
\usepackage{amstext}
\usepackage{amsthm}
\usepackage{amssymb}
\usepackage{graphicx}
\usepackage{enumitem}
\usepackage[letterpaper]{geometry}
\usepackage[pdfusetitle,
 bookmarks=true,bookmarksnumbered=false,bookmarksopen=false,
 breaklinks=false,pdfborder={0 0 1},backref=false,colorlinks=false]
 {hyperref}
\hypersetup{
 colorlinks=true,linkcolor=teal,citecolor=purple,linktocpage=true}

\makeatletter
\numberwithin{equation}{section}
\numberwithin{figure}{section}
\theoremstyle{plain}
\newtheorem{thm}{\protect\theoremname}[section]
\theoremstyle{definition}
\newtheorem{defn}[thm]{\protect\definitionname}
\theoremstyle{plain}
\newtheorem{lem}[thm]{\protect\lemmaname}
\theoremstyle{remark}
\newtheorem{rem}[thm]{\protect\remarkname}
\theoremstyle{plain}

\theoremstyle{remark}

\theoremstyle{plain}

\theoremstyle{remark}

\@ifundefined{date}{}{\date{}}
\usepackage{amsmath, tikz-cd, amssymb, placeins}
\DeclareRobustCommand*\cal{\@fontswitch\relax\mathcal}

\usetikzlibrary{calc}
\usetikzlibrary{decorations.pathmorphing}

\tikzset{curve/.style={settings={#1},to path={(\tikztostart)
    .. controls ($(\tikztostart)!\pv{pos}!(\tikztotarget)!\pv{height}!270:(\tikztotarget)$)
    and ($(\tikztostart)!1-\pv{pos}!(\tikztotarget)!\pv{height}!270:(\tikztotarget)$)
    .. (\tikztotarget)\tikztonodes}},
    settings/.code={\tikzset{quiver/.cd,#1}
        \def\pv##1{\pgfkeysvalueof{/tikz/quiver/##1}}},
    quiver/.cd,pos/.initial=0.35,height/.initial=0}

\tikzset{tail reversed/.code={\pgfsetarrowsstart{tikzcd to}}}
\tikzset{2tail/.code={\pgfsetarrowsstart{Implies[reversed]}}}
\tikzset{2tail reversed/.code={\pgfsetarrowsstart{Implies}}}
\tikzset{no body/.style={/tikz/dash pattern=on 0 off 1mm}}

\makeatother

\providecommand{\claimname}{Claim}
\providecommand{\corollaryname}{Corollary}
\providecommand{\definitionname}{Definition}
\providecommand{\lemmaname}{Lemma}
\providecommand{\propositionname}{Proposition}
\providecommand{\remarkname}{Remark}
\providecommand{\theoremname}{Theorem}

\begin{document}
\title{Brunnian links of 3-balls in the 4-sphere}

\author{Seungwon Kim}
\address{Sungkyunkwan University\\Suwon, Gyeonggi, 16419 Republic of Korea}
\email{seungwon.kim@skku.edu}

\author{Gheehyun Nahm}
\address{Department of Mathematics, Princeton University, Princeton, New Jersey
08544, USA}
\email{gn4470@math.princeton.edu}

\author{Alison Tatsuoka}
\address{Department of Mathematics, Princeton University, Princeton, New Jersey 08540, USA}
\email{at8451@princeton.edu}

\thanks {SK was supported by National Research Foundation of Korea (NRF) grants funded by the Korean government (MSIT) (No.\ 2022R1C1C2004559). GN was partially supported by the ILJU Academy and Culture
Foundation, the Simons collaboration \emph{New structures in low-dimensional
topology}, and a Princeton Centennial Fellowship. AT was partially supported by an NSF Graduate Research Fellowship and by the NSF under Grant No. DMS-1928930
while in residence at the Simons Laufer Mathematical Sciences Institute in Berkeley, California,
during the Spring 2026 semester.}

\begin{abstract}
For each integer $n\ge 2$, we construct infinitely many $n$-component Brunnian links of 3-balls in $S^4$. Our main tool is the third author's result on the existence
of splitting spheres for the trivial two-component link of $2$-spheres
in $S^{4}$; we also give a new proof of this.
\end{abstract}

\maketitle
%
% \tableofcontents{}

\section{\label{sec:Introduction}Introduction}

In a breakthrough work \cite{budney2021knotted3ballss4}, Budney and Gabai constructed infinitely many knotted $3$\nobreakdash-balls in $S^4$. In this paper, for each integer $n\ge 2$, we construct infinitely many $n$-component \emph{Brunnian links of $3$-balls} in $S^{4}$. More precisely:

\begin{defn}[Brunnian links of $3$-balls]
\label{def:brunnian-links}Let $n\ge 2$ and fix $n$ pairwise disjoint $3$-balls $B_{1}, \cdots, B_{n}$ in $S^4$. Let $K_i =\partial B_i$ and $K=K_1 \sqcup \cdots \sqcup K_n $. We refer to $B:=B_{1} \sqcup \cdots \sqcup B_{n}$ as the $n$-component \emph{unlink} of spanning $3$-balls for $K$.

Let $B_1 ', \cdots ,B_n '$ be another collection of pairwise disjoint $3$-balls such that $\partial B_i ' = K_i $. We say that the $n$-component link of spanning $3$-balls $B'=B'_{1}\sqcup\cdots\sqcup B'_{n}$ is \emph{Brunnian} if $B'$ and $B$ are not isotopic rel.\ $K$, but for each $i=1,\cdots,n$, the $(n-1)$-component links of spanning 3-balls $B\setminus B_{i}$ and $B'\setminus B_{i}'$ are isotopic
rel.\ $K\setminus K_{i}$.
\end{defn}

Our main result is the following:
\begin{thm}\label{thm:main-thm}
For each $n\geq2$, there exist infinitely many pairwise non-isotopic
$n$-component Brunnian links of 3-balls in $S^{4}$. 
\end{thm}

We prove Theorem~\ref{thm:main-thm} for all $n\ge 2$ in Theorem~\ref{thm:n-cpt-Brunnian}. For ease of exposition, we also give an alternative construction for the $n=2$ case in Theorem~\ref{thm:2-cpt Brunnian}, since the proof that these links are Brunnian is simpler than that of Theorem~\ref{thm:n-cpt-Brunnian}. The construction of Theorem~\ref{thm:2-cpt Brunnian} also generalizes to all $n\ge 2$; see Remark~\ref{rem:n2simple}.

We construct our Brunnian links of 3-balls by employing an infinite family of diffeomorphisms $\boldsymbol{\delta_k}$ of $S^1 \times B^3$ that Budney-Gabai used to construct knotted 3-balls in \cite{budney2021knotted3ballss4}. These diffeomorphisms are instances of \emph{barbell diffeomorphisms}; we review these constructions in Sections~\ref{sec:Barbells,-knotted-disks,}~and~\ref{sec:knotted-s3}.

The main tool that we use to distinguish the Brunnian links is the third author's result \cite{tatsuoka2025splittingspheresunlinkeds2s}
(Theorem~\ref{thm:another-proof-theorem}) on the existence of splitting
spheres for the trivial two-component link of $2$-spheres in $S^{4}$. We also present a new proof of Theorem~\ref{thm:another-proof-theorem},
which was motivated by a question asked by Peter Kronheimer at the
June 2025 ICTP conference; we thank him for the question.

\begin{defn}
\label{def:splitting-sphere}A splitting sphere $\Sigma$ for a 2-component link $K=K_1\sqcup K_2\subset S^n$ is an embedded $S^{n-1}$ in $S^n$ such that the $K_i$ are contained in different connected components of $S^n\setminus \Sigma$. Two splitting spheres for a link $K$ are isotopic if they are isotopic in $S^n$ rel.\ $K$.
\end{defn}

\begin{thm}[\cite{tatsuoka2025splittingspheresunlinkeds2s}; restated in Theorem~\ref{thm:another-proof}]
\label{thm:another-proof-theorem}There exist infinitely many pairwise
non-isotopic splitting 3-spheres for the trivial two-component link
of $2$-spheres in $S^{4}$.

\end{thm}

\begin{rem}[Unknotting and unlinking $3$-balls in the $5$-ball]Powell \cite{powellspanning} showed that any two $n$-component links of spanning $3$-balls for an $n$-component unlink in $S^4$ become isotopic rel.\ $\partial$ after pushing their interiors into $B^5$. His work generalizes a result of Hartman \cite{hartman_unknotting3balls}, who showed the same statement for spanning 3-balls for the unknot in $S^4$.
\end{rem}

\begin{rem}The $n=2$ case of Theorem~\ref{thm:main-thm} was independently proved recently by Niu \cite{niu2026brunnianspanning3disks2unlink}. Niu also uses barbell diffeomorphisms to construct his Brunnian links, but distinguishes them using different methods, namely a generalization of Budney and Gabai's $W_3$ invariant  \cite{budney2021knotted3ballss4}.
\end{rem}

\subsection{Organization}
In Section \ref{sec:Barbells,-knotted-disks,}, we review the construction
of barbell diffeomorphisms introduced by Budney and Gabai in \cite{budney2021knotted3ballss4}. In Section~\ref{sec:knotted-s3}, we present another proof of Theorem~\ref{thm:another-proof-theorem}
\cite{tatsuoka2025splittingspheresunlinkeds2s} (Theorem \ref{thm:another-proof}) on the existence of splitting spheres for unknotted 2-spheres in $S^4$, and review the relevant
tools from \cite{budney2021knotted3ballss4} needed for the proof. In Section~\ref{sec:brunnian}, we construct $n$-component
Brunnian links of 3-balls in $S^{4}$ for all $n\geq2$ (Theorem \ref{thm:n-cpt-Brunnian}).

\subsection{Conventions}
Manifolds are oriented, and maps between manifolds are orientation preserving. If $X$ is an oriented manifold, $\overline{X}$ is $X$ equipped with the opposite orientation.

$A\approx B$ means that $A$ and $B$ are isotopic. Diffeomorphisms and isotopies are rel.\ $\partial$ if they fix the boundary pointwise. Similarly, if $F$ is a set, then diffeomorphisms and isotopies are rel.\ $F$ if they fix $F$ pointwise. 

Unless otherwise specified, when we draw objects in $S^{4}$, we view
$S^{4}=B^{4}\cup S^{3}\times[-1,1]\cup \overline{B^{4}}$, and we draw
their intersection with the $S^{3}\times0$ ``time slice'' of $S^{4}$,
where we view this time slice $S^{3}\times0$ as $\mathbb{R}^{3}$ with
a point at infinity. Some objects do not lie fully in this
time slice: notably, we often consider $2$- and $3$-spheres that
intersect $S^3 \times 0$ in equatorial $1$- and $2$-spheres, respectively.
In these cases, we imagine the rest of these spheres being capped
off by $2$- and $3$-disks, respectively, in the nearby time slices
$S^{3}\times\pm\varepsilon$. (This description of $2$-spheres can
also be thought of as a banded unlink diagram with no
bands.)

\subsection*{Acknowledgements}
GN thanks Peter Ozsv\'{a}th for his continuous support and helpful
discussions. AT thanks Dave Gabai for his invaluable mentorship and advice.

\section{\label{sec:Barbells,-knotted-disks,}Barbell diffeomorphisms}

In this section, we review the construction of barbell diffeomorphisms; for more details, see Chapter 5 of \cite{budney2021knotted3ballss4}. Budney and Gabai construct interesting diffeomorphisms (\emph{``barbell
diffeomorphisms''}) of a $4$-manifold $X$ by first defining an
auxiliary $4$-manifold (the \emph{``model thickened barbell''})
$\mathcal{NB}:=S^{2}\times D^{2}\natural S^{2}\times D^{2}$ and an
interesting diffeomorphism (the \emph{``barbell map''}) $\beta:\mathcal{NB}\to\mathcal{NB}$
rel.\ $\partial$. Then, they specify an embedding of ${\cal NB}$
into $X$, and obtain a diffeomorphism of $X$ by pushing forward
$\beta$; this operation is called \emph{barbell implantation}.

\begin{figure}[h]
\begin{centering}
\includegraphics[scale=0.3]{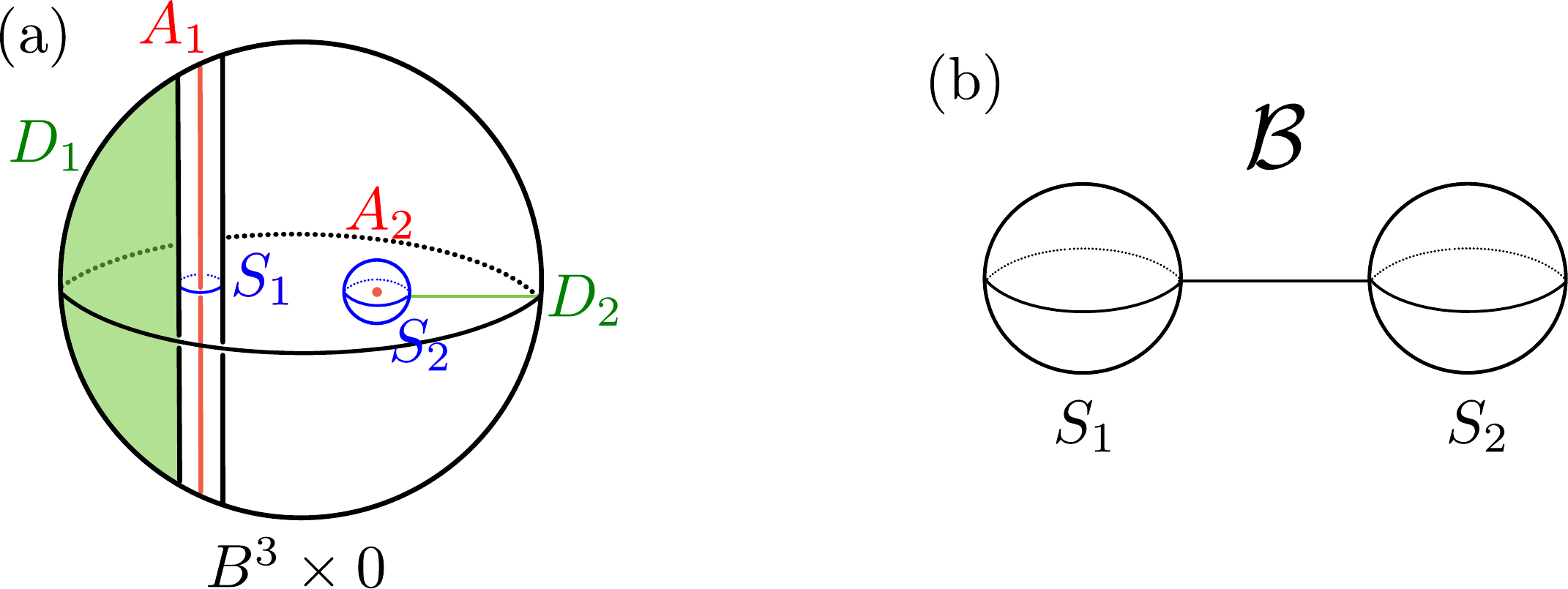}
\par\end{centering}
\caption{\label{fig:barbell}(a): The arcs $A_{1}$ and $A_{2}$, the 2-spheres
$S_{1}$ and $S_{2}$, and the 2-disks $D_{1}$ and $D_{2}$ in the
$B^{3}\times0$ slice of $B^{4}=B^{3}\times[-1,1]$. Note that $D_{1}$
appears fully in this 3-dimensional slice, while $D_{2}$ intersects
it in an arc; similarly, $S_{2}$ appears fully in this slice, while
$S_{1}$ intersects it in an equatorial $S^{1}$. (b): The barbell
$\mathcal{B}$. }
\end{figure}

\subsection{\label{subsec:barbell-map}The barbell map $\beta:\mathcal{NB}\to\mathcal{NB}$}

Consider two disjoint, properly embedded arcs $A_{1}$ and $A_{2}$
in $B^{4}$. Then $B^{4}\setminus(N(A_{1})\sqcup N(A_{2}))\cong S^{2}\times D^{2}\natural S^{2}\times D^{2}$.
The barbell map $\beta$ will be defined by applying the isotopy extension
theorem to a particular isotopy of $A_{1}$ in $B^{4}$ rel.\ $N(A_{2})$.
This isotopy is called \emph{arc-spinning}, and is given by \emph{spinning
$A_{1}$ positively around $A_{2}$} (or, equivalently, spinning $A_{1}$
positively around the meridian $2$-sphere $S_{2}$ of $A_{2}$),
which we now explain.

Viewing $B^{4}$ as $B^{3}\times[-1,1]$, suppose that $A_{1}$ appears
fully in the $B^{3}\times0$ slice of $B^{4}$ as a vertical arc,
and let $A_{2}$ be the arc $\mathrm{pt}\times[-1,1]$ that intersects
$B^{3}\times0$ in a point disjoint from $A_{1}$ (Figure~\ref{fig:barbell}~(a)).
Isotope $A_{1}$ in $B^{3}\times0$ by dragging it \emph{positively}
around the point of intersection $A_{2}\cap(B^{3}\times0)$ until
it returns back to where it started, keeping $\partial A_{1}$ fixed
throughout. We can arrange so that the neighborhood $\overline{N}(A_{1})$
begins and ends in the same position after performing this isotopy,
so that applying the ambient isotopy extension theorem, our isotopy
of $A_{1}$ rel.\ $\overline{N}(A_{2})$ is realized by a 1-parameter family
of diffeomorphisms of $B^{4}$ rel.\ $\overline{N}(A_{2})$ for which the
time-$1$ map is a diffeomorphism of $B^{4}$ that fixes $\overline{N}(A_{1})\sqcup \overline{N}(A_{2})\sqcup \overline{N}(\partial B^{4})$
pointwise. Cutting out $N(A_{1})\sqcup N(A_{2})$ from $B^{4}$
and taking the restriction of this time-$1$ map gives us our diffeomorphism
$\beta:\mathcal{NB}\rightarrow\mathcal{NB}$; note that $\beta$ fixes
the boundary $\partial(\mathcal{NB})$, as desired. Also note that $\beta$ is isotopic
rel.\ $\partial$ to the diffeomorphism of $\mathcal{NB}$ obtained by spinning $A_{2}$ \emph{negatively} around $A_{1}$ in $B^4$ (see \cite[Remark 5.4 ii)]{budney2021knotted3ballss4}).

\subsection{Barbell implantation}

Given an embedding $f:\mathcal{NB}\hookrightarrow X$, we can push
forward the barbell map $\beta\in\text{Diff}_{\partial}(\mathcal{NB})$
via the embedding $f$, and obtain a diffeomorphism, called $\beta_{f}\in\text{Diff}_{\partial}(X)$,
of $X$ rel.\ $\partial$. This operation is called \emph{barbell
implantation}, and the diffeomorphisms obtained in this way are called
\emph{barbell diffeomorphism}s. In this subsection, we recall the
information necessary to specify the barbell diffeomorphism $\beta_{f}$
up to isotopy rel.\ $\partial$.

Observe that $\mathcal{NB}$ deformation retracts onto the 2-complex
given by two 2-spheres connected by an arc (Figure~\ref{fig:barbell}~(b));
this 2-complex is called the \textit{barbell} and is denoted by $\mathcal{B}$.
By \cite[Remarks 5.12]{budney2021knotted3ballss4}, the barbell diffeomorphism
$\beta_{f}$ is determined, up to isotopy rel.\ $\partial$, by the
restriction $f|_{{\cal B}}$ (recall that our convention requires maps between manifolds to be orientation-preserving); in particular, $\beta_{f}$ does not depend
on the framing of the bar.
This motivates the following definition.
\begin{defn}[Barbells]
\label{def:barbell}A \emph{barbell} in a $4$-manifold
$X$ is a collection of the following:
\begin{enumerate}
\item two oriented, embedded $S^{2}$'s with trivial normal bundle (\emph{``cuffs''}),
and
\item an oriented, embedded interval $[0,1]$ that connects the two $S^{2}$'s
(\emph{``bar''}).
\end{enumerate}
We require that the cuffs are disjoint, and that the bar intersects
the cuffs only at its endpoints.
\end{defn}

By the above discussion, a barbell $\eta$
in $X$ uniquely determines (up to isotopy rel.\ $\partial$) the
corresponding barbell diffeomorphism in ${\rm Diff}_{\partial}(X)$.
Note that the orientation on the bar is only used to tell the two
cuffs apart. We denote the barbell diffeomorphism
given by implanting a barbell $\eta$ as $\boldsymbol{\eta}\in{\rm Diff}_{\partial}(X)$.
\begin{lem}
\label{lem:reverseori}Let $\eta$ be a barbell in an oriented $4$-manifold
$X$. Then, reversing the orientation of one cuff or the bar changes
the corresponding barbell diffeomorphism $\boldsymbol{\eta}$ to its inverse.
\end{lem}

\begin{proof}
This follows by construction (compare \cite[Theorem 5.6]{budney2021knotted3ballss4},
which is the ``harder case'': reversing the orientation of the bar).

All three statements follow from the following two observations: (1)
the diffeomorphism obtained by spinning an arc positively around a
sphere is the inverse (up to isotopy rel.\ $\partial$) of the diffeomorphism
obtained by spinning the same arc \emph{negatively} around the same
sphere. (2) The diffeomorphism on ${\cal NB}$ obtained by spinning
the first arc positively around the second sphere is isotopic rel.\ $\partial$
to the the diffeomorphism on ${\cal NB}$ obtained by spinning the
second arc negatively around the first sphere.
\end{proof}
Hence, the eight choices of the orientations of the cuffs and the
bar result in two (potentially) distinct barbell implantations up to isotopy rel.\ $\partial$, and
one is the inverse of the other up to isotopy rel.\ $\partial$. Note that it is important to be careful about the orientations for our proof of Theorem~\ref{thm:another-proof}. 

\section{\label{sec:knotted-s3}Barbells and knotted $S^{3}$'s}

In this section, we give an alternate proof (Theorem \ref{thm:another-proof}) of the third author's theorem \cite{tatsuoka2025splittingspheresunlinkeds2s}, which says that there are infinitely many pairwise non-isotopic splitting spheres of the trivial two-component 2-link $S^2 \sqcup S^2$ in $S^4$. We distinguish the splitting spheres by showing that they lift to non-isotopic non-separating 3-spheres in $S^1 \times S^3$, the double branched cover of $S^4$ along $S^2 \sqcup S^2$. That the 3-spheres are non-isotopic in $S^1 \times S^3$ is detected by Budney and Gabai's $W_3$ invariant for non-separating 3-spheres in $S^1 \times S^3$.

Before proving Theorem \ref{thm:another-proof}, we first give a brief review of their constructions of knotted 3-spheres and their $W_3$ invariant that they use to distinguish them. We note that we state only the necessary ingredients for the proof of Theorem \ref{thm:another-proof}, though for the reader's convenience we try to give precise references for their statements and proofs in \cite{budney2021knotted3ballss4}.

\begin{defn}
Let ${\rm Emb}_{0}(S^{3},S^{1}\times S^{3})$ be the space of embeddings
of $S^{3}$ into $S^{1}\times S^{3}$ that are homotopic to the\emph{
standard embedding ${\rm pt}\times S^{3}$ of $S^{3}$}. Let ${\rm Diff}_{0}(S^{1}\times S^{3})$
be the space of diffeomorphisms $S^{1}\times S^{3}\to S^{1}\times S^{3}$
that are homotopic to the identity.
\end{defn}

Let $\varphi\in{\rm Diff}_{0}(S^{1}\times S^{3})$. Budney and Gabai
define the $W_{3}$ invariant of $\varphi$, which only depends on
$\varphi({\rm pt}\times S^{3})$, by (1) associating to $\varphi$ some element $[\varphi]$
in the group $\pi_{2}{\rm Emb}_{\partial}(I,S^{1}\times B^{3};I_{0})$ (which depends on some additional choices), and (2) defining a homomorphism
from $\pi_{2}{\rm Emb}_{\partial}(I,S^{1}\times B^{3};I_{0})$ to
a more tractable vector space over $\mathbb{Q}$ (and showing
that the image of $[\varphi]$ does not depend
on the additional choices made in step (1)) \cite[Corollary 7.18 and Theorem 8.5]{budney2021knotted3ballss4}.

\begin{thm}[{\cite[Proposition 3.3, Corollary 7.18, Theorem 8.5, and Corollary 8.6]{budney2021knotted3ballss4}}]
\label{thm:w3thm}There exists a homomorphism
\begin{equation}
W_{3}:\pi_{0}{\rm Diff}_{0}(S^{1}\times S^{3})\to\mathbb{Q}[t_{1}^{\pm1},t_{3}^{\pm1}] \bigg/\left\langle t_{1}^{\alpha-\beta}t_{3}^{-\beta}-t_{1}^{\alpha}t_{3}^{\alpha-\beta}=t_{1}^{\beta-\alpha}t_{3}^{-\alpha}-t_{1}^{\beta}t_{3}^{\beta-\alpha}\right\rangle \label{eq:w3}
\end{equation}
that factors through the map $\pi_{0}{\rm Diff}_{0}(S^{1}\times S^{3}) \to \pi_{0} {\rm Emb}_{0}(S^{3},S^{1} \times S^{3})$ induced by 
\[
{\rm Diff}_{0}(S^{1}\times S^{3})\to{\rm Emb}_{0}(S^{3},S^{1}\times S^{3}):\varphi\mapsto\varphi({\rm pt}\times S^{3}).
\]
\end{thm}

\begin{figure}[h]
\begin{centering}
\includegraphics{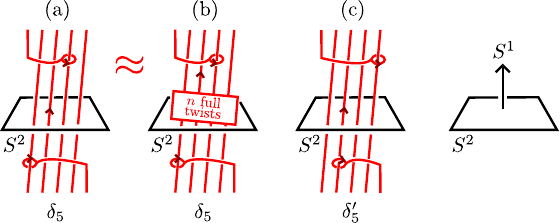}
\par\end{centering}
\caption{\label{fig:delta5prime}Barbells $\delta_{5}$ and $\delta_{5}'$; only their intersections with $S^{1}\times S^{2}\times0$ are drawn.
The arrows on the cuffs specify the orientation of the oriented intersection of the cuffs with $S^{1}\times S^{2}\times0$. The arrows on the bars specify the orientation of the bars.
Note (Remark~\ref{rem:twist-barbell}) that the barbells drawn in (a) and (b) are isotopic.}
\end{figure}

Budney and Gabai \cite[Section 7 and Theorem 8.1]{budney2021knotted3ballss4}
also compute the $W_{3}$ invariant for certain classes of barbell
diffeomorphisms of $S^{1}\times S^{3}$ \cite[Definition 6.11]{budney2021knotted3ballss4},
called the \emph{twisted $\theta_{k}$ implantations} $\boldsymbol{\theta_{k}}(\overrightarrow{v},\overrightarrow{w})$ for $\overrightarrow{v},\overrightarrow{w}\in \mathbb{Z}^{k-1}$.
We refer the reader to \cite[Definition 6.11]{budney2021knotted3ballss4}
for the general definition, and focus on the special cases 
\begin{equation}
\boldsymbol{\delta_{k}}:=\boldsymbol{\theta_{k}}((0,0,\cdots,0,1),(0,0,\cdots,0,1,0)),\ \boldsymbol{\delta_{k}'}:=\boldsymbol{\theta_{k}}((0,0,\cdots,0,1,0),(0,0,\cdots,0,1)).\label{eq:deltadeltaprime}
\end{equation}

We view $S^{1}\times S^{3}=(S^{1}\times B^3)\cup(S^{1}\times S^{2}\times[-1,1])\cup(\overline{S^{1}\times B^3})$. The barbells $\delta_{k}$ and $\delta_{k}'$ are contained in $S^{1}\times S^{2}\times [-1,1]$; their intersections with $S^{1}\times S^{2}\times 0$ are drawn in Figure~\ref{fig:delta5prime} for $k=5$ (here, we view $S^{2}=\mathbb{R}^{2}\cup\{\infty\}$; in fact they are contained in $S^{1}\times \mathbb{R}^{2}\times [-1,1]$).

In general, the bars of $\delta_{k}$ and $\delta_{k}'$ loop around
the $S^{1}$ direction $k$ times. The top cuff of $\delta_{k}$ links
the second strand from the right, and the bottom cuff of $\delta_{k}$
links the first strand from the left. The top cuff of $\delta_{k}'$
links the first strand from the right, and the bottom cuff of $\delta_{k}'$
links the second strand from the left.
\begin{thm}
\label{thm:w3comp}$W_{3}(\boldsymbol{\delta_{4}}),W_{3}(\boldsymbol{\delta_{5}}),\cdots$ are
linearly independent. Also, we have $W_{3}(\boldsymbol{\delta_{k}})=-W_{3}(\boldsymbol{\delta_{k}'})$.
\end{thm}

\begin{proof}
The first part is \cite[Theorem 8.3]{budney2021knotted3ballss4}.
The second part also follows from \cite[Section 7 and Theorem 8.1]{budney2021knotted3ballss4},
where they compute the $W_{3}$ invariant for all the twisted $\theta_{k}$
implantations. For the reader's convenience, we give precise, minimal
number of references: first, Budney and Gabai define a family of elements
$F_{k}(p,q)\in\pi_{2}{\rm Emb}_{\partial}(I,S^{1}\times B^{3};I_{0})$
in \cite[Definition 7.16]{budney2021knotted3ballss4}, but their precise
definitions are not needed for this paragraph. By Equation~(\ref{eq:deltadeltaprime})
and \cite[Theorem 8.1]{budney2021knotted3ballss4}, we have $[\boldsymbol{\delta_{k}}]=F_{k}(k-1,k-2)$
and $[\boldsymbol{\delta_{k}'}]=F_{k}(k-2,k-1)$. Since $F_{k}(p,q)=-F_{k}(q,p)$
by \cite[Proof of Corollary 7.18]{budney2021knotted3ballss4}, we
have $W_{3}(\boldsymbol{\delta_{k}})=-W_{3}(\boldsymbol{\delta_{k}'})$.
\end{proof}
\begin{rem}
\label{rem:twist-barbell}Introducing any number of full twists to
the bars of $\delta_{k}$ as in Figure~\ref{fig:delta5prime} (b)
results in an isotopic barbell, since we are in $4$ dimensions and
so the twists can be undone.
\end{rem}

\begin{figure}[h]
\begin{centering}
\includegraphics{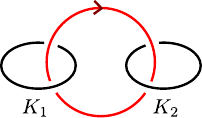}
\par\end{centering}
\caption{\label{fig:s1d3emb}An embedding of an $S^{1}$ (red) in the complement
of a trivial two-component $2$-link $K_1 \sqcup K_2 \subset S^{4}$. Only the intersection with the equatorial $S^3 \times 0 \subset  B^{4}\cup S^{3}\times[-1,1]\cup \overline{B^{4}} = S^4 $ is drawn.}
\end{figure}

Recall that the barbell $\delta _k $ is contained in $(S^{1}\times B^3)\cup(S^{1}\times S^{2}\times[-1,1]) \subset S^1 \times S^3 $. By identifying $(S^{1}\times B^3)\cup(S^{1}\times S^{2}\times[-1,1]) \cong S^1 \times B^3$, let us also view the barbell $\delta _k $ as a barbell contained in $S^1 \times B^3 $ and the barbell diffeomorphism $\boldsymbol{\delta _k}$ as a diffeomorphism of $S^1 \times B^3 $ rel.\ $\partial$.

Now we are ready to state and prove Theorem \ref{thm:another-proof}. Let $K=K_{1}\sqcup K_{2}$ be a trivial two-component $2$-link in $S^{4}$. The red circle in Figure~\ref{fig:s1d3emb} describes an embedding of $\iota :S^{1}\times B^3 \hookrightarrow S^{4}\setminus N(K)$ up to isotopy and framing. Let $\boldsymbol{\beta_{k}}$ be the diffeomorphism of $S^4$ rel.\ $K$ given by pushing forward $\boldsymbol{\delta_{k}} \in \mathrm{Diff}_\partial (S^1 \times B^3 )$ along $\iota$. Then, $\boldsymbol{\beta_{k}}$ is given by the barbell $\beta_k := \iota(\delta _k) \subset S^4 \setminus N(K)$. Note that by Remark~\ref{rem:twist-barbell}, the isotopy rel.\ $\partial N(K)$ class of $\beta_{k}$ does not depend on the framing of the red circle, and hence the isotopy rel.\ $K$ class of $\boldsymbol{\beta_{k}}$ does not either.

\begin{figure}[h]
\begin{centering}
\includegraphics{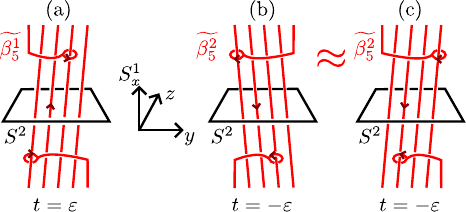}
\par\end{centering}
\caption{\label{fig:delta5deck}View $S^{1}\times S^{3}$ as $(S^{1}\times B^3)\cup(S^{1}\times S^{2}\times[-1,1])\cup(\overline{S^{1}\times B^3})$.
(a) and (b): The two lifts of the barbell $\beta_{5}$ to $S^{1}\times S^{3}$;
their intersections with the time slices $S^{1}\times S^{2}\times\pm\varepsilon$
are drawn. (c): An isotopic copy of the barbell from (b)}
\end{figure}

\begin{thm}[\cite{tatsuoka2025splittingspheresunlinkeds2s}]
\label{thm:another-proof}Let $\Sigma$ be the standard splitting
sphere of $K=K_{1}\sqcup K_{2}$ in $S^{4}$. Then, for all $k,\ell\ge4$,
$k\neq\ell$, the splitting spheres $\boldsymbol{\beta_{k}}\Sigma$ and $\boldsymbol{\beta_{\ell}}\Sigma$
are nonisotopic rel.\ $K$.
\end{thm}

\begin{proof}Consider the double branched cover of $S^{4}$ along $K$, which is $S^{1}\times S^{3}$. We distinguish $\Sigma$, $\boldsymbol{\beta_{k}}\Sigma$, and $\boldsymbol{\beta_{\ell}}\Sigma$ by distinguishing their lifts in $S^1 \times S^3$ using Theorem~\ref{thm:w3comp}. Let $\widetilde{\beta_{k}^{1}},\widetilde{\beta_{k}^{2}}$
be the two lifts of the barbell $\beta_{k}$ to $S^{1}\times S^{3}$.
Then, $\boldsymbol{\widetilde{\beta_{k}^{2}}}\circ \boldsymbol{\widetilde{\beta_{k}^{1}}}$ is a lift of $\boldsymbol{\beta_{k}}$.
Since the two lifts of $\Sigma$ to $S^{1}\times S^{3}$ are isotopic to $S_{std}^{3}:=\{\ast\}\times S^{3}$ and it with the opposite orientation, the two lifts of $\boldsymbol{\beta_{k}}\Sigma$ are isotopic to $\boldsymbol{\widetilde{\beta_{k}^{2}}} \boldsymbol{\widetilde{\beta_{k}^{1}}}S^3 _{std}$ and it with the opposite orientation. Hence, we are left to show that $S_{std}^{3}$, $\boldsymbol{\widetilde{\beta_{k}^{2}}\widetilde{\beta_{k}^{1}}}S_{std}^{3}$, 
and $\boldsymbol{\widetilde{\beta_{\ell}^{2}}\widetilde{\beta_{\ell}^{1}}}S_{std}^{3}$
are pairwise nonisotopic.

Let $\rho:S^{1}\times S^{3}\to S^{1}\times S^{3}$ be the nontrivial
deck transformation of the double branched cover of $S^{4}$ along $K$.
Let us describe $\rho$ more explicitly: view $S^{1}\times S^{3}$
as $(S^{1}\times B^3)\cup(S^{1}\times S^{2}\times[-1,1])\cup(\overline{S^{1}\times B^3})$.
Then, $\rho$ swaps $S^{1}\times B^3$ and $\overline{S^{1}\times B^3}$.
For $S^{1}\times S^{2}\times[-1,1]$, view $S^{1}=\mathbb{R}/\mathbb{Z}$
and let it have coordinate $x$, view $S^{2}=\mathbb{R}^{2}\cup\{\infty\}$
and let it have coordinates $y,z$, and let $[-1,1]$ have coordinate
$t$. Then $\rho$ acts on $S^{1}\times S^{2}\times[-1,1]$ as $(x,y,z,t)\mapsto(-x,y,z,-t)$.

Hence, we see that the barbells $\widetilde{\beta_{k}^{1}},\widetilde{\beta_{k}^{2}}$
are as in Figure~\ref{fig:delta5deck}~(a)~and~(b). In particular, $\widetilde{\beta_{k}^{1}}$
is isotopic to $\delta_{k}$, and $\widetilde{\beta_{k}^{2}}$ is
isotopic (rotate by $\pi$ in the $yz$ direction) to the barbell
in Figure~\ref{fig:delta5deck}~(c), which is obtained from $\delta_{k}'$
by reversing the orientations of the two cuffs and the bar. Hence, the barbell diffeomorphisms 
$\boldsymbol{\widetilde{\beta_{k}^{2}}}$ and $\boldsymbol{\delta_{k}'}$ are inverses of each other by Lemma~\ref{lem:reverseori}, and so we have 
\[
W_{3}(\boldsymbol{\widetilde{\beta_{k}^{2}}})=-W_{3}(\boldsymbol{\delta_{k}'})=W_{3}(\boldsymbol{\delta_{k}})=W_{3}(\boldsymbol{\widetilde{\beta_{k}^{1}}})
\]
by Theorem~\ref{thm:w3comp}. Therefore, we have
\[
W_{3}(\boldsymbol{\widetilde{\beta_{k}^{2}}}\circ\boldsymbol{\widetilde{\beta_{k}^{1}}})=2W_{3}(\boldsymbol{\delta_{k}}),
\]
and so $S_{std}^{3}$, $\boldsymbol{\widetilde{\beta_{k}^{2}}\widetilde{\beta_{k}^{1}}}S_{std}^{3}$, 
and $\boldsymbol{\widetilde{\beta_{\ell}^{2}}\widetilde{\beta_{\ell}^{1}}}S_{std}^{3}$
are pairwise nonisotopic by Theorems~\ref{thm:w3thm} and \ref{thm:w3comp}.
\end{proof}

\section{\label{sec:brunnian}Brunnian links of $3$-balls}

In this section, we first construct two component Brunnian links
of $3$-balls (Theorem~\ref{thm:2-cpt Brunnian}) as a warmup. Then,
we consider the general case $n\ge 2$ and construct $n$-component Brunnian
links of $3$-balls (Theorem~\ref{thm:n-cpt-Brunnian}) using a construction motivated
by Bing doubling. We detect their linking using the splitting spheres of \cite{tatsuoka2025splittingspheresunlinkeds2s} (Theorem~\ref{thm:another-proof}).

In this section, we construct diffeomorphisms of $S^4$ rel.\ $F$ for various $F\subset S^4$ as follows, similarly to the discussion right before Theorem~\ref{thm:another-proof}. We first construct embeddings $c: S^1 \hookrightarrow S^4 \setminus N(F)$, which in turn describe embeddings of $S^1 \times B^3$ in $S^4 \setminus N(F)$ up to isotopy and framing. Then, we view $\boldsymbol{\delta _k }$ as a diffeomorphism of $S^1 \times B^3$ rel.\ $\partial$ and pushforward $\boldsymbol{\delta _k }$ (resp.\ $\boldsymbol{\delta _k }^{-1}$) along these embeddings. Recall that the resulting diffeomorphisms, up to isotopy rel.\ $F$, only depend on the embedding $c$ of the core $S^1$; hence we refer to this diffeomorphism as the diffeomorphism given by \emph{pushing forward $\boldsymbol{\delta _k }$ (resp.\ $\boldsymbol{\delta _k }^{-1}$) along $c$}.

\begin{figure}[h]
\begin{centering}
\includegraphics{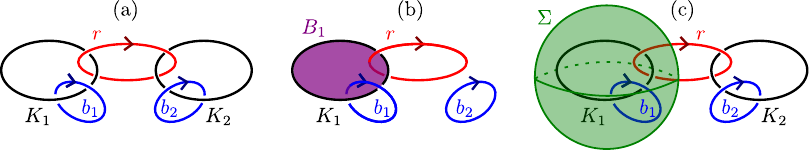}
\par\end{centering}
\caption{\label{fig:brunnian-red-blue}(a): A trivial two-component $2$-link
$K_{1}\sqcup K_{2}$ in $S^{4}$, a red circle $r$ and two blue circles $b_1 ,b_2$
in $S^{4}\setminus(K_{1}\sqcup K_{2})$. Pushing forward $\boldsymbol{\delta_{k}}$
along $r$ and $\boldsymbol{\delta_{k}}^{-1}$ along $b_1$ and $b_2$ gives rise to two-component Brunnian
links of $3$-balls in $S^{4}$. (b): If we ignore $K_{2}$, then
only $r$ and $b_1$ intersect $B_{1}$, and the corresponding 
diffeomorphisms cancel out. Note that $B_{1}$ intersects the $S^{3}\times0$
time slice in the purple 2-disk. (c): (a) together with the splitting sphere
$\Sigma=\partial N(B_{1})$, which intersects the $S^3\times 0$ slice in the green 2-sphere. Note that the blue circles do not intersect
$\Sigma$.}
\end{figure}

\begin{thm}
\label{thm:2-cpt Brunnian}Let $K=K_{1}\sqcup K_{2}$ be a trivial
2-link in $S^{4}$. Then there exists an infinite family of 2-component links of 3-balls $B_{1}\sqcup B_{2}$ with $\partial B_1 = K_1$ and $\partial B_2 = K_2 $ such that
for any two distinct links $B_{1}\sqcup B_{2}$ and $B_{1}'\sqcup B_{2}'$ in the family,
$B_1 \sqcup B_2\not\approx B_1'\sqcup B_2'$ rel.\ $K_{1}\sqcup K_{2}$, but $B_{1}\approx B_{1}'$
rel $K_{1}$ and $B_{2}\approx B_{2}'$ rel.\ $K_{2}$.
\end{thm}

\begin{proof}
Choose 3-balls $B_{1}$ and $B_{2}$ spanning $K_{1}$ and $K_{2}$ in $S^{4}$.
Let $\boldsymbol{r_k }$ be the diffeomorphism of $S^{4}$ rel.\ $K$ given by the pushing forward $\boldsymbol{\delta_{k}}$ along the red circle $r$ of Figure~\ref{fig:brunnian-red-blue}~(a).
Similarly, let $\boldsymbol{b_{k,1}}$ (resp.\ $\boldsymbol{b_{k,2}}$) be the
the pushforward of $\boldsymbol{\delta_{k}}^{-1}$ along the blue circle $b_1$ (resp.\ $b_2$) of Figure~\ref{fig:brunnian-red-blue}~(a).
Let $\boldsymbol{\gamma_{k}}:=\boldsymbol{r_k} \circ \boldsymbol{b_{k,1}} \circ \boldsymbol{b_{k,2}}$. We will show that $\{B_1 \sqcup B_2 \}\sqcup\{\boldsymbol{\gamma_{k}}(B_{1})\sqcup \boldsymbol{\gamma_{k}}(B_{2})\}_{k\geq4}$
is our desired infinite family.

First, since $b_2$ is disjoint from $B_1$, we have $\boldsymbol{\gamma_{k}}(B_{1})\approx \boldsymbol{r_k}\circ \boldsymbol{b_{k,1}}(B_{1})$ rel.\ $K$. If we ignore $K_2$, then $r$ and $b_1$ are isotopic rel.\ $K_1$. Hence, we have $\boldsymbol{r_k}\circ \boldsymbol{b_{k,1}}(B_{1})\approx B_{1}$ rel.\ $K_{1}$. The same argument shows $\boldsymbol{\gamma_{k}}(B_{2})\approx B_{2}$
rel $K_{2}$.

Now we will show that $\boldsymbol{\gamma_{k}}(B_{1})\sqcup \boldsymbol{\gamma_{k}}(B_{2})\not\approx B_{1}\sqcup B_{2}$
rel.\ $K$. Suppose that they were isotopic. Consider a small regular
neighborhood $\overline{N}(B_{1})$ of $B_{1}$ in $S^{4}$. Then
$\overline{N}(B_{1})\cong B^{4}$, and so $\Sigma:=\partial\overline{N}(B_{1})\cong S^{3}$
is a splitting sphere for the two-component link $K$. If $\boldsymbol{\gamma_{k}}(B_{1})\sqcup\boldsymbol{\gamma_{k}}(B_{2})\approx B_{1}\sqcup B_{2}$
rel.\ $K$, then the subspaces $\boldsymbol{\gamma_{k}}(\overline{N}(B_{1}))$
and $\overline{N}(B_{1})$ of $S^{4}$ are isotopic via an ambient
isotopy of $S^{4}$ rel.\ $K$; in particular, we must have that
$\boldsymbol{\gamma_{k}}(\Sigma)\approx\Sigma$ rel.\ $K$. On the other hand,
since the blue circles can be isotoped rel.\ $K$ to be disjoint
from $\Sigma$, then
$\boldsymbol{\gamma_{k}}(\Sigma)\approx \boldsymbol{r_{k}}(\Sigma)$ rel.\ $K$. However, $\boldsymbol{r_k}$ is the diffeomorphism $\boldsymbol{\beta_k}$ from Theorem~\ref{thm:another-proof} and so 
$\boldsymbol{r_{k}}(\Sigma)\not\approx\Sigma$ rel.\ $K$ by Theorem~\ref{thm:another-proof},
giving us our contradiction.
Since varying $k$ produces non-isotopic
splitting spheres $\boldsymbol{\beta_{k}}(\Sigma)$, the same proof shows that
$\boldsymbol{\gamma_{k}}(B_{1})\sqcup\boldsymbol{\gamma_{k}}(B_{2})\not\approx\boldsymbol{\gamma_{\ell}}(B_{1})\sqcup\boldsymbol{\gamma_{\ell}}(B_{2})$
rel.\ $K$ for $k\neq\ell$, $k,\ell\geq4$, giving us our infinite
family of 2-component Brunnian 3-ball links.
\end{proof}
\begin{figure}[h]
\begin{centering}
\includegraphics{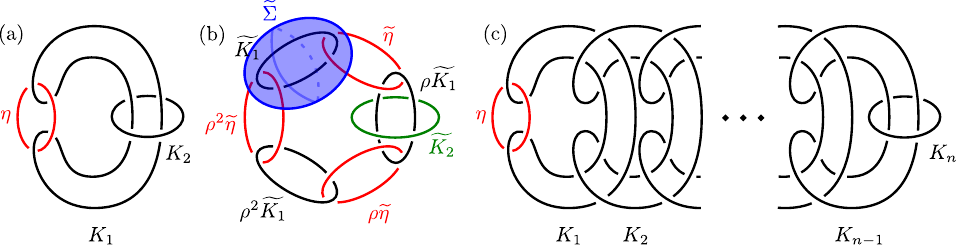}
\par\end{centering}
\caption{\label{fig:brunnian-bing}(a), (c): Pushing forward $\boldsymbol{\delta_{k}}$ along
$\eta$ gives rise to Brunnian links of $3$-balls
in $S^{4}$. (b): the $3$-fold cyclic branched cover of $S^{4}$
over $K_{2}$ and the lift $\widetilde{\Sigma}$ of the splitting
$3$-sphere $\Sigma=\partial N(B_{1})$ that splits $\widetilde{K_{1}}$
from $\rho\widetilde{K_{1}}\sqcup\rho^{2}\widetilde{K_{1}}$.}
\end{figure}

Now we prove the general case.
\begin{thm}
\label{thm:n-cpt-Brunnian}Let $K=K_{1}\sqcup \cdots \sqcup K_{n}$ be
a trivial $n$-component 2-link in $S^{4}$. Then there exists an infinite family of $n$-component links of 3-balls $B_{1}\sqcup \cdots \sqcup B_{n}$ with
$\partial B_{i}=K_{i}$ such that for any two distinct links $B=B_{1}\sqcup\cdots\sqcup B_{n}$
and $B' =B_{1}'\sqcup\cdots\sqcup B_{n}'$ in the family, $B$ and $B'$ are nonisotopic rel.\ $K$,
but for each $i=1,\cdots,n$, the $(n-1)$-component links of 3-balls $B\setminus B_{i}$
and $B'\setminus B_{i}'$ are isotopic rel.\ $K\setminus K_{i}$.
\end{thm}

\begin{proof}
Consider the circle $\eta$ in the complement of
the $n$-component unlink $K=K_{1}\sqcup\ldots\sqcup K_{n}$ in $S^{4}$
given in Figure~\ref{fig:brunnian-bing}~(c). Let $B=B_{1}\sqcup\cdots\sqcup B_{n}$,
where $B_{i}$ is a $3$-ball with $\partial B_{i}=K_{i}$. Let $\boldsymbol{\eta_{k}}$ be the diffeomorphism of $S^4$ rel.\ $K$ given by pushing forward  $\boldsymbol{\delta_{k}}$ along $\eta$. We will show that $\{B\}\sqcup \{\boldsymbol{\eta_{k}}(B)\}_{k\geq4}$ is our desired infinite family.

First note that by construction, if we ignore any of the $K_{i}$'s,
then we can ``unlink'' $\eta$ from $K\setminus K_{i}$, i.e.\ we can isotope $\eta$ rel.\ $K\setminus K_{i}$ into a $4$-ball disjoint from $K\setminus K_{i}$, and hence we can further isotope  $\eta$ rel.\ $K\setminus K_{i}$ to be disjoint from $B$.
Thus, the $(n-1)$-component 3-ball links $\boldsymbol{\eta_{k}}(B)\setminus \boldsymbol{\eta_{k}}(B_{i})$
and $B\setminus B_{i}$ are isotopic rel.\ $K\setminus K_{i}$ for
$i=1,...,n$.

Now, we induct on $n$ and show that $B, \boldsymbol{\eta_{k}}(B), \boldsymbol{\eta_{\ell }}(B)$ are pairwise nonisotopic rel.\ $K$ for distinct $k,\ell \ge 4$. The base case of the induction is $n=2$: see Figure~\ref{fig:brunnian-bing}~(a).  Just as in the proof of Theorem~\ref{thm:2-cpt Brunnian}, to show
that $\boldsymbol{\eta_{k}}(B_{1})\sqcup\boldsymbol{\eta_{k}}(B_{2})\not\approx B_{1}\sqcup B_{2}$
rel.\ $K$, it suffices to show that $\boldsymbol{\eta_{k}}(\Sigma)\not\approx\Sigma$
in $S^{4}$ rel.\ $K$ for the splitting sphere $\Sigma:=\partial \overline{N}(B_{1})$
for $K$. So suppose that $\boldsymbol{\eta_{k}}(\Sigma)\approx\Sigma$. We consider
the 3-fold cyclic branched cover of $S^{4}$ over $K_{2}$: see Figure~\ref{fig:brunnian-bing}~(b)
(note that this branched cover is $S^{4}$ again). Let $\rho$ be
the deck transformation corresponding to the meridian of $K_{2}$,
let $\widetilde{K_{2}}$ be the lift of $K_{2}$, and let $\widetilde{\eta}$
and $\widetilde{K_{1}}$ be a lift of $\eta$ and $K_{1}$, respectively,
as in the figure.

Let $\boldsymbol{\widetilde\eta_{k}}$ be the composition of the three diffeomorphisms given by pushing forward $\boldsymbol{\delta_{k}}$ along the lifts $\widetilde{\eta}$, $\rho\widetilde{\eta}$, and $\rho^{2}\widetilde{\eta}$ of $\eta$. Then, $\boldsymbol{\widetilde\eta_{k}}$ is a lift of $\boldsymbol{\eta_{k}}$.
Let $\widetilde{\Sigma}$ be the lift
of $\Sigma$ that splits $\widetilde{K_{1}}$ from $\rho\widetilde{K_{1}}\sqcup\rho^{2}\widetilde{K_{1}}$.
If $\boldsymbol{\eta_{k}}(\Sigma)\approx\Sigma$ rel.\ $K$, then in the 3-fold
branched cyclic cover, there exists some $i\in \{0,1,2\}$ such that $\boldsymbol{\widetilde{\eta}_{k}}(\widetilde{\Sigma})\approx\rho^{i}\widetilde{\Sigma}$
rel.\ $\bigsqcup_{i=0}^{2}\rho^{i}\widetilde{K_{1}}$. Among the
$\rho^{i}\widetilde{\Sigma}$'s for $i=0,1,2$, $\widetilde{\Sigma}$
is the only one that splits $\widetilde{K_{1}}$ from $\rho\widetilde{K_{1}}\sqcup\rho^{2}\widetilde{K_{1}}$,
and so we must have $\boldsymbol{\widetilde{\eta}_{k}}(\widetilde{\Sigma})\approx\widetilde{\Sigma}$
rel.\ $\bigsqcup_{i=0}^{2}\rho^{i}\widetilde{K_{1}}$. However, we claim that $\boldsymbol{\widetilde{\eta}_{k}}(\widetilde{\Sigma})$
is not even isotopic to $\widetilde{\Sigma}$ rel.\ $\widetilde{K_{1}}\sqcup\rho\widetilde{K_{1}}$. Just as in the
proof of Theorem~\ref{thm:2-cpt Brunnian}, if we ignore $\rho^{2}\widetilde{K_1}$, then $\rho\widetilde{\eta}$
and $\rho^{2}\widetilde{\eta}$ can be isotoped rel.\ $\widetilde{K_{1}}\sqcup\rho\widetilde{K_{1}}$
to not intersect $\widetilde{\Sigma}$. Hence if $\boldsymbol{\beta_{k}}$ is
the diffeomorphism of $S^{4}$ rel.\ $\widetilde{K_{1}}\sqcup\rho\widetilde{K_{1}}$
given by pushing forward $\boldsymbol{\delta_{k}}$ along $\widetilde{\eta}$,
we have 
\[
\boldsymbol{\widetilde{\eta}_{k}}\widetilde{\Sigma}\approx\boldsymbol{\beta_{k}}\widetilde{\Sigma}\not\approx\widetilde{\Sigma}\ {\rm rel.\ }\widetilde{K_{1}}\sqcup\rho\widetilde{K_{1}}
\]
by Theorem~\ref{thm:another-proof} and so $\boldsymbol{\eta_{k}}(\Sigma)\not\approx\Sigma$ rel.\ $K$. The same
argument shows that the pairs $\boldsymbol{\eta_{k}}(B_{1})\sqcup\boldsymbol{\eta_{k}}(B_{2})$
and $\boldsymbol{\eta_{\ell}}(B_{1})\sqcup\boldsymbol{\eta_{\ell}}(B_{2})$ are also nonisotopic
rel.\ $K$ for $k\neq\ell$.

\begin{figure}[h]
\begin{centering}
\includegraphics{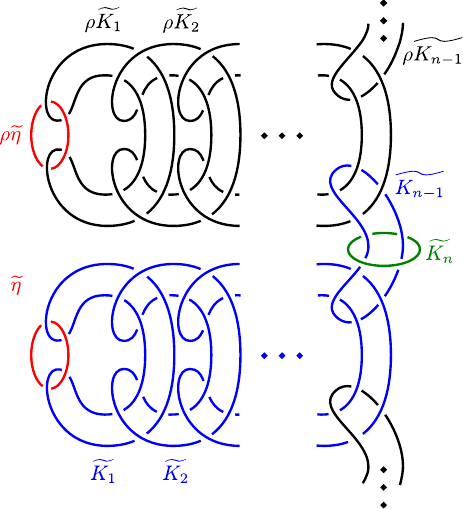}
\par\end{centering}
\caption{\label{fig:brunnian-cyclic}An $m$-fold ($m\ge2$) cyclic branched cover of Figure~\ref{fig:brunnian-bing}~(c) over $K_{n}$}
\end{figure}

Let us carry out the induction step; let $n\ge 3$. Let us consider an $m$-fold ($m\ge2$) cyclic branched
cover of $S^{4}$ along $K_{n}$ (which again is $S^{4}$) as in Figure~\ref{fig:brunnian-cyclic}. Let $\rho$ be the deck transformation
that corresponds to the meridian of $K_{n}$. Then we see $m$ lifts
$\widetilde{\eta},\rho\widetilde{\eta},...,\rho^{m-1}\widetilde{\eta}$
of $\eta$.
Let $\boldsymbol{\widetilde {\eta_k} }$ be the diffeomorphism given by the composition of the pushforwards of $\boldsymbol{\delta_{k}}$ along $\widetilde{\eta},\rho\widetilde{\eta},...,\rho^{m-1}\widetilde{\eta}$.
If $\boldsymbol{\eta_{k}}(B)\approx B$ rel.\ $K$, then in particular
we have $\boldsymbol{\eta_{k}}(B\setminus B_{n})\approx B\setminus B_{n}$ rel.\ $K$;
therefore in the branched cyclic cover, if
$\widetilde{B_{1}},\ldots,\widetilde{B_{n-1}}$ are the lifts of $B_{1},\ldots,B_{n-1}$ for which $\partial\widetilde{B}_{i}$ is the lift $\widetilde{K}_{i}$
of $K_{i}$ drawn in Figure~\ref{fig:brunnian-cyclic}, then 
\begin{equation}
\boldsymbol{\widetilde{\eta}_{k}}(\widetilde{B}_{1})\sqcup\ldots\sqcup\boldsymbol{\widetilde{\eta}_{k}}(\widetilde{B}_{n-1})\approx\widetilde{B}_{1}\sqcup\ldots\sqcup\widetilde{B}_{n-1}\ {\rm rel.}\ \widetilde{K}_{1}\sqcup\ldots\sqcup\widetilde{K}_{n-1}.\label{eq:reduct-n-1-brunnian}
\end{equation}

On the other hand, if we ignore the lift of $K_n$ and the lifts of $K_{1},\cdots, K_{n-1}$ other than $\widetilde{K_1}, \cdots \widetilde{K_{n-1}}$, then we see that the configuration of $\widetilde{\eta}$
and $\widetilde{K_1}\sqcup  \cdots \sqcup \widetilde{K_{n-1}}$ is the same as the configuration of the $\eta$ and $K_1\sqcup  \cdots \sqcup K_{n-1}$ that we used
to construct the $(n-1)$-component Brunnian links of $3$-balls.
By the induction hypothesis, Equation~(\ref{eq:reduct-n-1-brunnian}) does
not hold. Therefore, $\boldsymbol{\eta_{k}}(B)\not\approx B$ rel.\ $K$. The same argument
shows that $\boldsymbol{\eta_{k}}(B)\not\approx\boldsymbol{\eta_{\ell}}(B)$ rel.\ $K$ for
$k\neq\ell$.
\end{proof}
\begin{rem}
In the proof of Theorem~\ref{thm:n-cpt-Brunnian} for
$n=2$, it is possible to show the following by carefully keeping
track of orientations: consider the $m$-fold cyclic branched cover
of $S^{4}$ over $K_{2}$ for any $m\ge2$, let $\widetilde{K_{1}}$
be a lift of $K_{1}$, and let $\widetilde{B_{1}}$ and $\widetilde{B_{1}^{\eta}}$
be the lifts of $B_{1}$ and $\boldsymbol{\eta_{k}}(B_{1})$, respectively, such
that their boundary is $\widetilde{K_{1}}$. Then, $\widetilde{B_{1}}$
and $\widetilde{B_{1}^{\eta}}$ are not isotopic rel.\ $\partial$.
This computation is similar to that of our proof of Theorem~\ref{thm:another-proof}:
in $S^{4}\setminus N(\widetilde{K_{1}})\cong S^{1}\times B^3$,
we have $\widetilde{B_{1}^{\eta}}=\boldsymbol{\delta_{k}'}^{-1} \circ \boldsymbol{\delta_{k}}(\widetilde{B_{1}})$.
\end{rem}

\begin{figure}[h]
\begin{centering}
\includegraphics{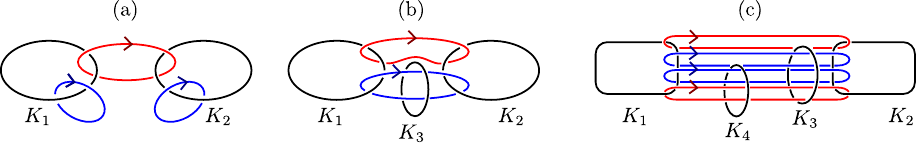}
\par\end{centering}
\caption{\label{fig:brunnian-red-blue-general}Pushing forward $\boldsymbol{\delta_{k}}$ along
the red circles and $\boldsymbol{\delta_{k}}^{-1}$ along the blue
circles give rise to Brunnian links of $3$-balls in
$S^{4}$.}
\end{figure}

\begin{rem}\label{rem:n2simple}
Note that we could have constructed $n$-component Brunnian 3-ball
links using the $2$-component links of Theorem~\ref{thm:2-cpt Brunnian}
directly as our base case, and generalizing the construction drawn
in Figure~\ref{fig:brunnian-red-blue-general}~(b)~and~(c). We
found the construction of Theorem~\ref{thm:n-cpt-Brunnian}, motivated
by Bing doubling, easier for exposition, and leave the alternative
proof to the interested reader. (To show that the links of $3$-balls are nonisotopic, induct on $n$ and consider the double branched cover of $S^4$ along $K_n$.)
\end{rem}

\bibliographystyle{amsalpha}
\bibliography{bib}

\end{document}